\documentclass[12pt]{amsart}
\usepackage{amsmath}
\usepackage{amsfonts}
\usepackage{amssymb}
\usepackage{geometry}
\usepackage{amsthm}
\usepackage{mathrsfs}
\usepackage{url}
\usepackage{hyperref}
\usepackage{cite}
\usepackage{amsrefs}
\usepackage{indentfirst}
\usepackage{epsfig}
\usepackage{graphicx}
\numberwithin{equation}{section}
\newtheorem{theorem}{Theorem}[section]

\newtheorem{pro}[theorem]{Proposition}

\newtheorem{cor}[theorem]{Corollary}

\theoremstyle{remark}

\begin{document}
\title[]{Remark on a determinant involving prime numbers}
\author[]{Huan Xiao}
\address{School of Artificial Intelligence, Zhuhai City Polytechnic, Zhuhai, China}
\email{xiaogo66@outlook.com}
\curraddr{}
\thanks{}

\begin{abstract}
Pausinger recently investigated a special determinant involving prime numbers. In this short note we point out that this type of determinants was already known in linear algebra and its computation is unrelated to prime numbers.
\end{abstract}

\maketitle

Let $ p_{n} $ be the $ n $th prime number. Recently Pausinger \cite{p} investigated a special determinant involving prime numbers. It is the  determinant of the $ n\times n $ matrix whose diagonal contains the first $  n$ prime numbers and all other elements are ones and thus
\begin{equation*}
D_{n}=\begin{vmatrix}
2 & 1& \cdots &1\\ 
1 & 3& \cdots &1\\ 
\vdots &\vdots &\ddots & \vdots\\
1 & 1& \cdots &p_{n}\\ 
\end{vmatrix}.
\end{equation*}
It is the sequence A067549 of The On-Line Encyclopedia of Integer Sequences \cite{e} and the first few values of $ D_{n} $ are $ 2, 5, 22, 140, 1448, 17856 $. Pausinger \cite{p} relates this sequence to a 
counting problem and obtains a recurrence relation on $ D_{n} $, see \cite[Lemma 2]{p}.\par
Here in this short note we point out that this type of determinants like $ D_{n} $ was already known in linear algebra. Although at this moment I am unable to give a specific reference, the computation of $ D_{n} $ is standard and the value of it is well known. Indeed we have
\begin{pro} 
Let $ a_{1}a_{2}\cdots a_{n}\neq 0 $, then
\begin{equation*}
d_{n}=\begin{vmatrix}
1+a_{1} & 1& \cdots &1\\ 
1 & 1+a_{2}& \cdots &1\\ 
\vdots &\vdots &\ddots & \vdots\\
1 & 1& \cdots &1+a_{n}\\ 
\end{vmatrix}=\left(1+\sum_{k=1}^{n}\dfrac{1}{a_{k}}\right)a_{1}a_{2}\cdots a_{n}.
\end{equation*}
\end{pro}
\begin{proof}
This result is already known in linear algebra and for the completeness we give a proof here. Let $ r_{i} $ and $ c_{i} $ denote the $ i $-th row and $ i $-th column of the matrix in question. The first step is making the elementary transformations $ r_{i}-r_{1} $ for $ i=2,\cdots ,n $, then we have
\begin{equation}
d_{n}= \begin{vmatrix}
1+a_{1} & 1& \cdots &1\\ 
-a_{1} & a_{2}&  \\ 
\vdots & &\ddots & \\
-a_{1} & &  &a_{n}\\ 
\end{vmatrix}.
\end{equation}
The second step is making the following elementary transformations for $ i=2,\cdots ,n $
\begin{equation*}
c_{1}+\dfrac{a_{1}}{a_{i}}c_{i}.
\end{equation*}
The resulting determinant is
\begin{equation}\label{0.2} 
d_{n}=\begin{vmatrix}
b & 1& \cdots &1\\ 
0 & a_{2}&  \\ 
\vdots & &\ddots & \\
0 & &  &a_{n}\\ 
\end{vmatrix}
\end{equation}
where 
\begin{equation}
b=1+a_{1}+a_{1}\sum_{k=2}^{n}\dfrac{1}{a_{k}}=a_{1}\left(1+\sum_{k=1}^{n}\dfrac{1}{a_{k}}\right).
\end{equation}
By \eqref{0.2} we have
\begin{equation}
d_{n}=\left(1+\sum_{k=1}^{n}\dfrac{1}{a_{k}}\right)a_{1}a_{2}\cdots a_{n}.
\end{equation}
\end{proof}

Now taking $ a_{k}=p_{k}-1 $ we have
\begin{cor}
\begin{equation}
D_{n}=\begin{vmatrix}
2 & 1& \cdots &1\\ 
1 & 3& \cdots &1\\ 
\vdots &\vdots &\ddots & \vdots\\
1 & 1& \cdots &p_{n}\\ 
\end{vmatrix}=\left(1+\sum_{k=1}^{n}\dfrac{1}{p_{k}-1}\right)\prod_{k=1}^{n}\left(p_{k}-1\right).
\end{equation}
\end{cor}

\end{document}